\documentclass[11pt,a4paper,dvipsnames]{amsart}
\usepackage{amsmath,amsthm,amssymb,amsfonts,mathrsfs}
\usepackage{mathtools}
\usepackage{hyperref} 
\usepackage[all]{xy}
\usepackage[utf8]{inputenc}
\usepackage{enumitem}
\usepackage{dsfont}
\usepackage[only,llbracket,rrbracket]{stmaryrd}

\newtheorem{thm}{Theorem}[section]
\newtheorem{mainthm}{Theorem}
\newtheorem{prop}[thm]{Proposition}
\newtheorem{lma}[thm]{Lemma}
\newtheorem{coro}[thm]{Corollary}
\theoremstyle{definition}
\newtheorem{dfn}[thm]{Definition}

\newtheorem{rmk}[thm]{Remark}
\newtheorem{ex}[thm]{Example}

\newcommand{\Z}{\mathds{Z}}
\newcommand{\V}{\VV}

\newcommand{\R}{\mathds{R}}
\newcommand{\C}{\mathds{C}}
\newcommand{\Q}{\mathds{Q}}
\newcommand{\Pj}{\PP}
\newcommand{\A}{\mathds{A}}

\newcommand{\VV}{{\mathds V}}

\newcommand{\ff}{\mathbf{f}}
\newcommand{\calodim}{\operatorname{codim}}

\newcommand{\Pic}{\operatorname{Pic}}
\newcommand{\Cl}{\operatorname{Cl}}

\usepackage{xcolor}

\newcommand{\ZZ}{{\mathds Z}}

\newcommand{\CC}{\mathds C}

\newcommand{\PP}{\mathds P}

\newcommand{\cald}{{\mathcal D}}
\newcommand{\cale}{{\mathcal E}}
\newcommand{\calf}{{\mathcal F}}
\newcommand{\calg}{{\mathcal G}}

\newcommand{\cali}{{\mathcal I}}

\newcommand{\caln}{{\mathcal N}}
\newcommand{\calo}{{\mathcal O}}

\newcommand{\calt}{{\mathcal T}}

\newcommand{\calv}{{\mathcal V}}

\newcommand{\ox}{{\mathcal O}_X}

\newcommand{\p}[1]{{\mathds{P}^{#1}}}
\newcommand{\op}[1]{{\mathcal O}_{\PP^{#1}}}

\newcommand{\supp}{\operatorname{Supp}}

\newcommand{\inext}{{\mathcal E}{\it xt}}

\newcommand{\Hom}{\operatorname{Hom}}

\DeclareMathOperator{\caloker}{coker}

\DeclareMathOperator{\im}{Im}

\DeclareMathOperator{\Dv}{Div}
\DeclareMathOperator{\rk}{{rk}}

\newcommand{\lra}{\longrightarrow}

\newcommand{\diag}{\mathrm{diag}}

\usepackage[top=2.6cm, bottom=2.6cm, left=2.7cm, right=2.7cm]{geometry}


\title[Toric logarithmic vector field and foliations]{Logarithmic vector fields and \\ foliations on toric varieties}

\author{Daniele Faenzi}
\address{Daniele Faenzi, Institut de Mathématiques de Bourgogne, UMR 5584, CNRS \& Université de Bourgogne, F-21000 Dijon, France}
\email{daniele.faenzi@u-bourgogne.fr}

\author{Marcos Jardim}

\address{Marcos Jardim, Instituto de Matemática, Estatística e Computação Científica, Universidade Estadual de Campinas (UNICAMP), Rua Sérgio Buarque de Holanda 651,
13083-859, Campinas, SP, Brazil}

\email{jardim@unicamp.br}

\author{William D. Montoya}

\address{William D. Montoya, Instituto de Matemática, Estatística e Computação Científica, Universidade Estadual de Campinas (UNICAMP), Rua Sérgio Buarque de Holanda 651,
13083-859, Campinas, SP, Brazil}

\email{wmontoya@ime.unicamp.br}

\thanks{M. J. is partially supported by the CNPQ grant number 305601/2022-9, the FAPESP Thematic Project 2018/21391-1 and the FAPESP-ANR project 2021/04065-6.
D. F. partially supported by FanoHK ANR-20-CE40-0023, SupToPhAG/EIPHI ANR-17-EURE-0002, Bridges ANR-21-CE40-0017, JSPS S24043.
W. M. acknowledges support from FAPESP postdoctoral grants number 2019/23499-7 and 2023/01360-2. 
We all benefited from the CAPES/COFECUB Project number 88887.191919/2018 -- Ma 926/19.
}
\begin{document}

\begin{abstract}
We introduce a toric version of the sheaf of logarithmic vector fields along a divisor of a simplicial toric variety. The notion is also relevant for algebraically independent families of polynomials in the Cox ring. We provide a generalization of the Saito criterion for the freeness of the toric logarithmic sheaf. We explain the relationship between this sheaf and the usual sheaf of logarithmic vector fields and the connection with holomorphic foliations on toric varieties.
\end{abstract}

\maketitle

\sloppy

\section{Introduction}

The study of vector fields or derivations tangent to some reduced divisor $D$ in a complex variety $X$ is a classical object in algebra, geometry, and algebraic geometry that has been studied for decades. Over the affine space $\A^n$, writing $f$ for the equation of the divisor $D$, the logarithmic derivations form a finitely generated module over the polynomial ring $S=\C[x_1,\ldots,x_n]$, or equivalently a sheaf over $\calo_{\A^n}$, denoted by $\calt_{\A^n}\langle D \rangle$, which is identified with the sheaf of Jacobian syzygies, namely, $\calt_{\A^n}\langle D \rangle$ is the kernel of the gradient $\bar \nabla(f)$, seen as a map $\calo_{\A^n}^{\oplus n} \to \calo_{D}$. Indeed, logarithmic derivations $\theta$ are defined by the condition that $f$ divides $\theta(f)$, which is to say that $\theta(f)$ vanishes modulo the equation of $D$.

In the influential paper \cite{saito:logarithmic}, K. Saito observed that for certain divisors the set of logarithmic derivations admits a basis, meaning that they form a free module over the algebra of holomorphic or polynomial functions. Furthermore, Saito also provided a simple and effective criterion that characterizes such special divisors, thus initiating a rich new area of research within complex/algebraic geometry and commutative algebra. 

The main goal of this paper is to define and study a toric version of the sheaf of logarithmic vector fields, letting $f$ be a square-free element of the Cox ring $S$ of a simplicial toric variety $X$ defined by a fan $\Sigma$ and considering the gradient $\nabla(f)$ with respect to the variables $x_1,\ldots,x_r$ defining torus-invariant divisors $D_1,\ldots,D_r$ associated with the $r$ rays 
$\Sigma(1)$ of $\Sigma$.
Writing $\bar \nabla (f)$ for the gradient of $f$, taken modulo $f$, we get the (extended) toric sheaf of logarithmic vector fields
\[
\calt_\Sigma \langle D \rangle := \ker (\bar \nabla(f)), \qquad \mbox{with} \qquad \bar \nabla(f) : \bigoplus_{1 \le i \le r }\calo_X(D_i) \to \calo_D(D).
\]
Note that this gives back $\calt_\Sigma\langle D \rangle = \calt_{\A^n}\langle D \rangle$ if $X=\A^n$. More generally, this sheaf is an extension of the usual sheaf $\calt_X \langle D\rangle$ by $\calo_X^{\oplus \rho}$, where $\rho = r-n$ is the rank of the Néron--Severi group of $X$, namely (see Proposition 
\ref{T vs TSigma}), we have a canonical exact sequence
\begin{equation}
\label{extension}    
0 \lra \calo_X^{\oplus \rho} \to \calt_\Sigma\langle D \rangle \lra \calt_{X}\langle D \rangle \lra 0. 
\end{equation}

Our first main result is a toric version of Saito's criterion. We call \textit{free} a coherent sheaf $\calf$ which is a direct sum of reflexive sheaves of rank $1$, in which case the divisors associated with the dual of such sheaves are called the exponents of $\calf$. For instance, $\calt_\Sigma \langle D\rangle$ is easily seen to be free for hypercube arrangements, see Example \ref{hypercube}.
Our result in this sense is the following (see Theorem \ref{thm:saito} for a comprehensive statement).

\begin{mainthm} \label{mthm1}
Let $X=X_\Sigma$ be a simplicial toric variety with no torus factors, and let $D =\V(f)\subset X$ be a reduced divisor. Then $\calt_\Sigma \langle D\rangle$ is free if and only if there is a free sheaf $\calf$ and a map $\nu: \calf \to \bigoplus_{1 \le i \le r }\calo_X(D_i)$
such that $\det (\nu|\epsilon) = cf$, with $c\in \C^*$.
If, in addition, $H^0(\calo_{X}(-\kappa_i))=0$ for $1 \le i \le n$, then 
$\calt_X \langle D\rangle$ is isomorphic to $\calf$.
\end{mainthm}

Here, $\epsilon$ refers to the \textit{Euler matrix} which is also responsible for the first morphism in sequence \eqref{extension}, see Definition \ref{coefficient matrix} for a precise setting including the coefficient matrix.
This result allows us to prove the freeness of sheaves of logarithmic derivations of what we call toric braid arrangements, see Proposition \ref{prop:braid}.
Also, it recovers and generalizes some examples of free divisors in Hirzebruch surfaces first given by Di Gennaro and Malaspina, see Example \ref{ex:34}. In a different direction, Napame studied in \cite{Napame} when the logarithmic tangent sheaf associated with an equivariant divisor in a projective toric variety is slope-stable for some choice of polarization. Our Saito criterion affords some results in this spirit, see Corollary \ref{corol stable} and Example \ref{example stable}.
\medskip

Turning to more general sheaves of logarithmic derivations, the first two named authors and Vallès proposed in \cite{FaenziJardimValle} a generalization of the notion of logarithmic tangent sheaves for algebraically independent $k$-tuples of homogeneous polynomials $\ff=(f_1,\ldots,f_k)$ in $n+1$ variables, interpolating the sheaf 
$\calt_{\PP^n}\langle \VV(f_1,\ldots,f_k)\rangle$. To be precise, setting $d_i:=\deg(f_i)$ for $1 \le i \le k$, we regard the Jacobian matrix $\nabla({\ff})$ as a morphism of sheaves
$$ \calo_{\Pj^n}(1)^{\oplus n+1}
\xlongrightarrow{\nabla({\ff})}
\bigoplus_{i=1}^k \calo_{\Pj^n}(d_i).$$
Then the \textit{logarithmic tangent sheaf} associated with $\ff$ is defined as the kernel of $\nabla({\ff})$. It can be interpreted as the intersection of all the logarithmic tangent sheaves for each of the divisors $\VV(f_i)$, see \cite[Lemma 2.5]{FaenziJardimValle}. For $k>1$, Muniz noticed that $\ker\big(\nabla({\ff})\big)$ is the tangent sheaf of a foliation of codimension $k-1$ on $\PP^n$, see \cite[Appendix]{FaenziJardimValle}. 
In this direction, given a $k$-tuple of homogeneous elements $\ff=(f_1,\ldots,f_k)$ in the Cox ring $S$ of a simplicial toric variety, we define a morphism of abelian groups $\ZZ^{\oplus k} \to \Cl(X)$ by sending $(a_1,\dots,a_k)$ to $\sum_i a_i\deg(f_i)$; let $q$ be the rank of this morphism and assume that $q<k$. When $X$ is projective, the degree vector $\deg(\ff)\in\Cl(X)^{\oplus k}$ can be seen as a set of $k$ points in $\p {q-1}$ and we say that $\deg(\ff)$ has the Cayley--Bacharach property if these points satisfy Cayley--Bacharach with respect to the hyperplane divisor, i.e., no $k-1$ points are contained in a hyperplane. This holds, for instance, when $q=1$, hence a fortiori when $\rho = 1$.

Again, we get a toric sheaf of logarithmic derivations (see Definition \ref{defn:toric-several} below) that can be interpreted as the intersection of toric logarithmic sheaves defined by the divisors $\V(f_i)$, see Lemma \ref{lem:intersection}. Our main result regarding this sheaf (see Theorem \ref{dist} for a more precise statement) is the following.

\begin{mainthm} \label{mthm2}
Let $X=X_\Sigma$ be a smooth projective simplicial toric variety.
Let $\ff$ be a sequence of $k \ge 2$ pairwise coprime algebraically independent homogeneous polynomials. Assume that $\ff$ has degree rank $q$ with $k-n < q < k$ and that $\deg(\ff)$ satisfies the Cayley--Bacharach condition. Then $\ff$ induces a foliation $\cald_{\ff}$ of codimension $k-q$ on $X$ whose singular scheme contains $\V(\ff)$.
\end{mainthm}

The paper is organized as follows. Section \ref{sec:pre} is dedicated to a brief review of simplicial toric varieties, the existence of a generalized Euler sequence, and the Euler formula. In Section \ref{sec:sheaves} we define the main characters of our work, namely the (extended) toric logarithmic sheaves on a toric variety for one polynomial. Section \ref{sec:saito} is mainly devoted to the notion of freeness for divisors on toric varieties and the proof of Theorem \ref{mthm1}. We then work out three examples, addressing cones on weighted projective space, toric braid arrangement, and invariant divisors. Section \ref{sec:dist} is about the case of several polynomials and their relationship with holomorphic distributions, leading up to the proof of Theorem \ref{mthm2}.

\smallskip

\noindent \textbf{Acknowledgements.} We thank Alan Muniz, Jean Vallès, and Maurício Corrêa for useful comments. Special thanks to Achim Napame for an important correction to a previous draft of this work.

\section{Preliminaries and notation} \label{sec:pre}

Let us recall some basic material on toric varieties. In this paper, the word \textit{variety} refers to a normal integral separated scheme of finite type over $\C$.

\subsection{Toric varieties}

A \textit{toric variety} is a variety $X$ containing a torus $T\simeq (\C^{*})^n$ as a Zariski open subset (thus $\dim(X)=n$) such that the action of $T$ on itself extends to an action $T\times X \to X$  of $T$ on $X$. 

\subsubsection{Cones and fans} Let us summarize some basic definitions and properties of cones and fans related to toric varieties.

\begin{dfn}
Let $M$ be a free abelian group of rank $n$. Let
$$N = \Hom_{\Z}(M, \Z), \qquad N_{\R} = N \otimes_{\Z} \R.$$
 \begin{enumerate}[label=\roman*)]
     \item A convex subset $\sigma \subset N_{R}$ is a rational $s$-dimensional  cone if there exist over $\R$,  $s$-elements
 $e_1,\dots, e_s \in N$ such that $$\sigma = \{\mu_1e_1 +\cdots+\mu_se_s \mid (\mu_1,\dots,\mu_s) \in \R_+^s\}.$$
     \item For any $i \in \{1,\ldots,s\}$, the generator $e_i$ is integral if for any non-negative rational number $l$ the product $l. e_i$ is in $N$ only if $l$ is an integer.
     \item Given two rational  cones $\sigma$, $\sigma'$ one says that $\sigma' $ is a face of $\sigma$ ($\sigma'< \sigma$)
     if the set of integral generators of $\sigma'$ is a subset of the set of integral generators of $\sigma$.
     \item A cone $\sigma$ is strongly convex if $\{0\}$ is a  face of $\sigma$.
     \item  A finite set $\Sigma = \{\sigma_1,\dots, \sigma_t\}$ of strongly convex rational  cones is called a fan if:
\begin{itemize}
    \item all faces of cones in $\Sigma$ are in $\Sigma$;
    \item if $\sigma,\sigma'\in \Sigma$ then $\sigma\cap \sigma'< \sigma$ and $\sigma\cap \sigma' <\sigma'$.
\end{itemize}     
\end{enumerate}
\end{dfn}

A fan $\Sigma\subset N_{\R}$ defines a \textit{toric variety} $X_\Sigma$ with torus $T_N=N\otimes_{\Z} \C^{*}$. By \cite[Corollary 3.1.8]{CoxLittleSchenck} if $X$ is a toric variety containing the torus $T_N$ as an affine open subset, then there exists a fan $\Sigma\subset N_{\R}$ such that $X\simeq X_\Sigma$.

We denote by $\Sigma(i)$ the $i$-dimensional cones of $\Sigma$. We call $\Sigma(1)$ the set of rays of $\Sigma$. Each $\varrho\in \Sigma(1)$ corresponds to an irreducible $T$-invariant Weil divisor $D_{\varrho}$ on $X_\Sigma$. Any Weil divisor $D$ is linearly equivalent to $\sum_{\varrho\in\Sigma(1)}a_{\varrho} D_{\varrho}$. We have an isomorphism:
\[
\Sigma(1) \simeq \Z^{r_\Sigma}, \qquad \mbox{where $r_\Sigma = \# \Sigma(1)$ is the \textit{toric rank} of $X_\Sigma$.}
\]

The divisors of the form $\sum_{\varrho\in \Sigma(1)}u_{\varrho}D_{\varrho} $ are precisely those divisors which are invariant under the torus action on $X_\Sigma$:
$$ \Dv_{T_N}(X_\Sigma)=\bigoplus_{\varrho \in \Sigma(1)} \Z D_{\varrho} \subset \Dv(X_\Sigma).$$
Here, $\Dv_{T_N}(X_\Sigma)$ is the group of $T_N$-invariant Weil divisors on $X_\Sigma$.

\subsubsection{The Cox ring} The \textit{Cox ring} of $X_\Sigma$ is the polynomial ring 
\[S=\C[x_{\varrho}\mid \varrho\in\Sigma(1)].\]

We refer to \cite{arzhantsev_derenthal_hausen_laface_2014} for an exhaustive study of this ring. Let us only mention here that the ring $S$ has a $\Cl(X_\Sigma)$-grading, which is described as follows. A monomial $x^a:=\prod_{\varrho\in\Sigma(1)}x_{\varrho}^{a_{\varrho}}\in S   $ is associated to the Weil divisor $D=\sum_{\varrho \in \Sigma(1)} a_\varrho D_\varrho$. Then $$\deg(x^a)=[D]\in \Cl(X_\Sigma).$$

\begin{rmk}
Writing $\rho=\rho(X_\Sigma)$ for the Picard rank of $X_\Sigma$, \cite[Theorem 4.2.1]{CoxLittleSchenck} gives
\[
\rho(X_\Sigma)=r_\Sigma-\dim(X_\Sigma).
\]
\end{rmk}

\subsubsection{Torus factors}
Let us recall here what we mean by a toric variety $X$ having a torus factor and the implications of this notion to the set of divisors on $X$.

\begin{dfn}
A toric variety $X$ has a \textit{torus factor} if it is equivariantly isomorphic to the product of a nontrivial torus and a toric variety of smaller dimension.
\end{dfn}

\begin{thm}[See {\cite[Corollary 3.3.10]{CoxLittleSchenck}}] The following are equivalent:
\begin{enumerate}[label=\roman*)]
    \item $X_\Sigma$ has no torus factors.
    \item Every morphisms $X_\Sigma\to \C^*$ is constant, i.e., $\Gamma(X_\Sigma,\calo_{X_\Sigma})^*=\C^*.$
    \item The minimal generators $u_{\varrho}$ of $\varrho \in \Sigma(1)$ span $N_{\R}$.
\end{enumerate}
\end{thm}

\begin{thm}[{\cite[Theorem 4.1.3]{CoxLittleSchenck}}]\label{gen}
One has the exact sequence
$$M\lra \Dv_{T_N}(X_\Sigma) \stackrel{\deg}{\lra} \Cl(X_\Sigma)\lra 0 . $$
Moreover, one has a short exact sequence
$$0\lra M\lra \Dv_{T_N}(X_\Sigma)\lra \Cl(X_\Sigma)\lra 0 . $$
if and only if $X_\Sigma$ has no torus factors.
\end{thm}

\subsubsection{Simplicial toric varieties}

Let us recall the notion of simplicial toric variety and its connection to orbifold singularities.

\begin{dfn}
A strongly convex rational polyhedral cone $\sigma\subset N_{\R}$ is simplicial if its minimal generators are linearly independent over $\R$ and we say that a fan $\Sigma$ is simplicial if every cone $\sigma$ in $\Sigma$ is simplicial.
\end{dfn}

\begin{thm}[Theorem 3.1.19 in \cite{CoxLittleSchenck}]
A toric variety $X_\Sigma$ is an orbifold, i.e., $X_\Sigma$  has only finite quotient singularities if and only if $\Sigma$ is simplicial.
\end{thm}

When $X_\Sigma$ is simplicial and without torus factors, it may be represented as 
$$X_\Sigma\simeq \C^r\setminus Z(\Sigma)/G $$
where $G=\Hom_{\Z}({\rm Cl}(\Sigma),\C^*)$ and $Z(\Sigma)=\VV(B(\Sigma))$ with $B(\Sigma)$ the irrelevant ideal, that is, 
$B(\Sigma):=\left(\prod_{\varrho \not \in \sigma(1)}x_{\varrho} \mid \sigma\in \Sigma\right)$. 

\begin{rmk}
A toric orbifold $X_\Sigma$ without torus factors has $\Cl(X_\Sigma)=0$ if and only if $X_\Sigma$ is the affine space. Indeed, $X_\Sigma$ is affine and smooth and by \cite[Example  1.2.21]{CoxLittleSchenck} we must have that $X_\Sigma$ is the product of the affine space and a torus. So, since  $X_\Sigma$ has no torus factors it must be the affine space. 
\end{rmk}

\begin{prop}[Proposition 4.2.7 in \cite{CoxLittleSchenck}]
For a given toric variety $X_\Sigma$, the following are equivalent:
\begin{enumerate}[label=\roman*)]
\item For every Weil divisor $D$ of $X_\Sigma$, there is an integer $m>0$ such that $mD$ is Cartier, namely, the variety $X_\Sigma$ is $\Q$-factorial.
\item The group $\Pic(X_\Sigma)$ has a finite index in $\Cl(X_\Sigma)$.
\item The fan $\Sigma$ is simplicial.
\end{enumerate}
\end{prop}

\subsection{Generalized Euler sequence in toric varieties} \label{generalized Euler}

Let us continue to use the notation introduced above, so $\Sigma$ is a simplicial fan, $X=X_\Sigma$ is the associated toric variety, $r=r_\Sigma=\# \Sigma(1)$ is the toric rank of $X$.

\begin{dfn}[Zariski 1-forms]
Let $j:U_0\hookrightarrow X$ the inclusion of the smooth locus of $X$. We define the sheaf of Zariski 1-forms as
$$\hat{\Omega}^1_{X}:=j_{*}\Omega^1_{U_0}. $$
\end{dfn}
In general $\hat{\Omega}^1_{X}$ may fail to be locally free but it is always reflexive.

\begin{thm}[{\cite[Theorem 12.1]{BatyrevCox}}]
Assume $X$ has no torus factors and $\Sigma$ is simplicial. Then there is an exact sequence 
$$ 0\lra \hat{\Omega}^1_{X} \lra \bigoplus_{i=1}^r\calo_{X}(-D_{i})\lra \Cl(X)\otimes_{\Z} \calo_{X}\lra 0   $$
and its dual sequence
$$0\lra \Cl(X)^\vee\otimes_{\Z} \calo_{X}\xrightarrow{\epsilon} \bigoplus_{i=1}^r\calo_{X}(D_{i})\lra \calt X\lra 0.$$
\end{thm}

Jaczewski showed in \cite{Jaczewski} that, if $X$ is a smooth variety that has a generalized Euler sequence, then $X$ is a toric variety. 

\begin{prop}[{\cite[Lemma 3.8]{BatyrevCox}}] \label{Euler relation}
If $\phi\in\Hom_{\Z}(\Cl(X),\Z)$ and $f\in H^0(\calo_{X}(\beta))$, then there exists a generalized Euler relation

$$\sum_{i=1}^r\phi([D_i])x_i\frac{\partial f}{\partial x_i}=\phi(\beta)\cdot f 
$$   
\end{prop}

\begin{rmk} If the rank of $\Cl(X)$ is strictly greater than 1, then there are more than 1 generalized Euler relations depending on the choice of $\phi$.
\end{rmk}

\section{Toric logarithmic tangent sheaves} \label{sec:sheaves}

Following Sernesi \cite{sernesi}, recall that the \textit{logarithmic tangent sheaf} $\calt_{X}\langle D \rangle$ of a reduced divisor $D$ in a variety $X$ is defined as the kernel of the composition:
\begin{equation} \label{eq:tau}
\tau_D : \calt_{X} \longrightarrow \calt_X|_D \longrightarrow \calo_D(D)
\end{equation}
where $\calo_D(D)$ is seen as the normal sheaf of $D$ in $X$ and the map $\calt_X|_D \lra \calo_D(D)$ is the usual epimorphism appearing in the normal sheaf sequence.

We remark that $\calt_{X}\langle D \rangle:=\ker(\tau_D)$ is always reflexive; in addition, if $X$ is non-singular, and $D$ is a normal crossing divisor, then $\calt_{X}\langle D \rangle$ is locally free.

\medskip

Here, we will define two classes of sheaves of logarithmic derivations that are adapted to the context of toric geometry and Jacobian matrices and explain how they are related to $\calt_{X}\langle D \rangle$.
Let $X=X_\Sigma$ be a $n$-dimensional simplicial toric variety and consider its Cox ring
$$ S=\C[x_1,\dots, x_r]=\bigoplus_{\alpha\in \Cl(X)}H^0(\calo_{X}(\alpha)).$$ 

Consider a nonzero homogogenous element $f \in S$ and set $D=\VV(f)$, $\beta = \deg(f) \in \Cl(X)$. We assume that $D$ is reduced. Then we have the Jacobian matrix:
\[\nabla(f)= \left(\frac{\partial f}{\partial x_1},\ldots,\frac{\partial f}{\partial{x_r}}\right) : \bigoplus_{1 \le i \le r} \calo_{X}(D_i)\lra \calo_{X}(\beta).
\]
Composing this with the projection $\calo_X \to \calo_D$ we get:
\[\overline \nabla(f)= \left(\frac{\partial f}{\partial x_1},\ldots,\frac{\partial f}{\partial{x_r}}\right) : \bigoplus_{1 \le i \le r} \calo_{X}(D_i)\lra \calo_{D}(\beta).
\]

We introduce the following two sheaves associated with $D$.

\begin{dfn} \label{defn:log-sheaves}
Let $X = X_\Sigma$ be a simplicial toric variety and let $D$ be a reduced divisor of $X$. The \textit{extended toric logarithmic sheaf} and the \textit{toric logarithmic sheaf} associated with a reduced divisor $D\subset X$ are respectively defined as:
$$ \calt_\Sigma\langle D \rangle := \ker\Big(\overline \nabla(f)\Big), \qquad {\rm and} \qquad \calt_\Sigma\langle D \rangle_0  := \ker\Big(\nabla(f)\Big). $$
\end{dfn}

The extended toric logarithmic tangent sheaf is related to the classical logarithmic tangent sheaf by a very simple exact sequence, which also explains our choice of terminology.

\begin{prop} \label{T vs TSigma}
For a reduced divisor $D \subset X$, we have:
\[ 0 \lra \Cl(X)^\vee\otimes_{\Z} \calo_{X} \lra \calt_\Sigma\langle D \rangle \lra \calt_{X}\langle D \rangle \lra 0. \]
In particular, $\calt_\Sigma\langle D \rangle $ is reflexive, and it is locally free whenever $\calt_{X}\langle D \rangle$ is locally free.
\end{prop}

\begin{proof} The morphism $\tau_D$ in display \eqref{eq:tau} can be written, in local coordinates, as the restriction to $D$ of the gradient of a defining equation $f$ of $D$. Up to lifting vector fields to local sections of $\bigoplus_{1 \le i \le r}\calo_X(D_i)$ via the Euler sequence, this map is globally described as the map $\overline \nabla (f)$, hence we get a commutative diagram
$$ \xymatrix{
 & 0 \ar[d] & 0 \ar[d] \\ 
 &         \Cl(X)^\vee\otimes_{\Z} \calo_{X} \ar[d] \ar@{=}[r] & \Cl(X)^\vee\otimes_{\Z} \calo_{X} \ar[d] \\
0 \ar[r] &   \calt_\Sigma\langle D \rangle \ar[r] \ar[d] & \bigoplus_{1 \le i \le r} \calo_X(D_i) \ar^-{\overline \nabla(f)}[r] \ar[d] & \calo_D(\beta) \ar@{=}[d]\\
0 \ar[r] &     \calt_{X}\langle D \rangle \ar[d]\ar[r] & \calt_X \ar^-{\tau_D}[r] \ar[d] & \calo_D(\beta) \\
& 0 & 0 } $$

The leftmost column gives the desired sequence. Since $\calt_{X}\langle D \rangle$ is reflexive, then $\calt_\Sigma\langle D \rangle$ is also reflexive because it is an extension of a reflexive sheaf. Similarly, if $\calt_{X}\langle D \rangle$ is locally free, then so is $\calt_\Sigma\langle D \rangle$.
\end{proof}

Next, we explain the relation between the two sheaves introduced in Definition \ref{defn:log-sheaves}.

\begin{prop} \label{T0}
Given a reduced divisor $D=\VV(f) \subset X$ with $\deg(f)=\beta\in\Cl(X)$, choose $\phi\in\Hom_{\Z}(\Cl(X),\Z)$ such that $\phi(\beta) \ne 0$. Then we have a splitting:
\[ \calt_\Sigma\langle D \rangle \simeq \calt_\Sigma\langle D \rangle_0 \oplus \calo_X. \]
\end{prop}

\begin{proof}
The image of map $\nabla(f)$ is the ideal sheaf of a closed subscheme $Z$ of codimension at least $2$ in $X$, tensored with $\calo_X(\beta)$. We denote this by $\cali_{Z/X}(\beta)$, and call $Z$ the Jacobian subscheme of $f$. Since $D$ is reduced, $Z$ has codimension at least $2$ in $X$. Also, $Z$ is contained in $D$ because, being $\phi(\beta) \ne 0$, the Euler relation (see Proposition \ref{Euler relation}) induces an exact sequence:
\begin{equation} \label{split T0}
0 \lra \calo_X \lra \cali_{Z/X}(\beta) \lra \cali_{Z/D}(\beta) \lra 0.
\end{equation}
Now, by definition, we have an inclusion $\calt_\Sigma\langle D \rangle_0 \subset \calt_\Sigma\langle D \rangle$ giving rise to an exact sequence:
\begin{equation} \label{Euler Jacobian}
0 \lra \calt_\Sigma\langle D \rangle_0  \lra \calt_\Sigma\langle D \rangle \lra \calo_X \lra 0,
\end{equation}
where the factor $\calo_X$ appears in view of \eqref{Euler Jacobian}. However, again the Euler relation shows that the map $\calo_X \to \bigoplus_{1 \le i \le r} \calo_X(D_i)$ given by $(x_1,\ldots,x_r)$ provides a splitting of \eqref{Euler Jacobian}.
\end{proof}

We conclude that $\calt_\Sigma\langle D \rangle_0$ is also reflexive, and it is locally free if and only if $\calt_\Sigma\langle D \rangle$ is locally free. Moreover, putting together Propositions \ref{T vs TSigma} and \ref{T0}, we obtain the exact sequence
\begin{equation} \label{eq:sqct0tx}
0 \lra \calo_{X}^{\oplus \rho(X)-1} \lra \calt_\Sigma\langle D \rangle_0 \lra \calt_{X}\langle D \rangle \lra 0,
\end{equation}
holding under the same hypotheses as in Proposition \ref{T0}. In particular, $\calt_\Sigma\langle D \rangle_0$ is isomorphic to
$\calt_{X}\langle D \rangle$ whenever $X$ has Picard rank 1.

\begin{ex}\label{ex:34}
Let $X=\PP\big(\op1(a_1)\oplus\cdots\oplus\op1(a_n)\big)\stackrel{\pi}{\to}\PP^1$ with $0\le a_1\le\cdots\le a_n$; set $a:=a_1 + \cdots + a_n$. This is a toric variety of dimension $n$; its Cox ring can be written as  $\C[x,y,u_1,\dots,u_n]$, where $x,y$ are coordinates for the base of the fibration $\pi$ and $u_1,\dots,u_n$ are coordinates in the fibres; in addition, $\Pic(X)=\Z\cdot H\oplus\Z\cdot F$ where  $F:=c_1(\pi^*\op1(1))$ and $H:=c_1(\calo_{\PP}(1))$ and $\calo_{\PP(1)}$ is the relative hyperplane bundle.

Let $D=L_1\cup\cdots\cup L_k$ be the divisor given by a union of fibers; note that $D=\V(f)$ for some polynomial $f\in\C[x,y,u_1,\dots,u_n]$ that depends only on $x$ and $y$ and $\deg(D)=kF$. We will now describe both $\calt_\Sigma\langle D \rangle_0$ and $\calt_{X}\langle D \rangle$.
The divisors $\VV(x)$ and $\VV(y)$ are of class $F$, while $\VV(u_j)$ is of class $H$ for $1 \le j \le n$. 
Then we have an exact sequence
$$ 0 \lra \calt_\Sigma\langle D \rangle_0 \lra 
\calo_X(F)^{\oplus 2} \oplus \calo_X(H)^{\oplus n} \xrightarrow{\nabla(f)}\ox(kF).$$
We note that, since $f$ only depends on the variables $x$ and $y$, the latter $n$ summands of the middle term must factor to the kernel of $\nabla(f)$, so 
$$ \calt_\Sigma\langle D \rangle_0 \simeq \calo_X(H)^{\oplus n}\oplus \ker\Big(\calo_X^{\oplus 2} \to \ox(kF)\Big); $$
The second summand must be a rank 1 reflexive sheaf, so it is the line bundle since $X$ is non-singular; computing degrees, we conclude that it must be $\calo_X((2-k)F)$. It follows that $\calt_\Sigma\langle D \rangle_0$ splits as a sum of line bundles and the exact sequence in display \eqref{eq:sqct0tx} becomes
\begin{equation} \label{eq:sqctxd}
0 \lra \ox \lra 
\calo_X(H)^{\oplus n}
\oplus \ox((2-k)F) \lra \calt_X\langle D \rangle \lra 0.
\end{equation}

Assuming now that $k\ge3$, we have that the morphism $\ox\to\ox((2-k)F)$ in the previous sequence must vanish, so $\ox((2-k)F)$ must be a summand of $\calt_X\langle D \rangle$. Additionally, the cokernel of the morphism $\ox \lra \calo_X^{\oplus n}(H)$ is precisely the relative tangent bundle $T_\pi X$ with respect to the base $\PP^1$. We then conclude that $\calt_X\langle D \rangle\simeq\ox((2-k)F)\oplus T_\pi X$.

Finally, when $n=2$, then $T_\pi X=\ox(2H-aF)$, so $\calt_X\langle D \rangle\simeq\ox((2-k)F)\oplus \ox(2H-aF)$, a fact observed by di Gennaro and Malaspina in \cite[Proposition 4.3]{dGM} under the condition that $k\ge a_2-a_1+1$.
\end{ex}

\section{The Saito criterion for toric logarithmic sheaves} \label{sec:saito}

An important class of divisors in the study of singularities and their unfoldings is that of free divisors. Here, we give a definition of freeness that is suitable for the toric setting. Let $X=X_\Sigma$ be a simplicial toric variety.

\begin{dfn} \label{defn:free}
A coherent sheaf $\calf$ is \textit{free} if it is the direct sum of rank 1 reflexive sheaves:
$$\calf \simeq   \bigoplus_{i=1}^s\calo_{X}(\alpha_i).$$
Here  $s$ is the rank of $\calf$ and, for $1 \le i \le s$,  $\alpha_i\in\Cl(X)$ is a class of a  Weil divisor.
In this case $(-\alpha_1,\ldots,-\alpha_t)$ are called the \textit{exponents} of $\calf$. Note that the exponents are in $\Cl(X)$.
\end{dfn}

The goal of this section is to give a simple, effective criterion, analogous to Saito's freeness criterion, to check whether the extended toric logarithmic tangent sheaf associated with a reduced divisor is free. Recall the map $\epsilon$ appearing in the Euler sequence, see \S \ref{generalized Euler}.

Before enunciating our criterion, we must understand the notion of homogeneous syzygies in the toric context.

\subsection{Jacobian syzygies}

Here we interpret the maps appearing in the Saito criterion as homogeneous syzygies. We do it for an algebraically independent family $\ff=(f_1,\ldots,f_k) \in S^k$.

\begin{dfn} A \textit{homogeneous syszygy} of $\nabla(\ff)$ is an $r$-tuple $\mu=(\mu_1,\dots, \mu_r) \in S^{\oplus r}$ with:
\[\nabla (\ff) \circ \mu=0, \qquad 
\deg(\mu_1) -\deg(x_1) =\cdots =\deg(\mu_r)-\deg(x_r).\]

Letting $\kappa \in \Cl(X)$ be $\deg(\mu_i)-\deg(x_i)$, for any given $i$ in $1,\ldots,r$, we get that  
$\deg(\mu_i)=\kappa+\deg(x_i)$.
Hence $\mu$ can be regarded as a morphism of sheaves:
$$\calo_{X}(-\kappa)\xlongrightarrow{\mu}  \bigoplus_{i=1}^r\calo_{X}(D_i) $$
\end{dfn}

The next lemma provides the elementary relationship between syzygies and global sections.

\begin{lma} For $\kappa\in \Cl(X)$, let $\operatorname{Syz}_{\kappa}(\nabla(\ff))$ be the syzygies of degree $\kappa$ for $\nabla(\ff)$. Then,
$$ H^0(\calt_\Sigma\langle \ff\rangle_0(\kappa))\simeq \operatorname{Syz}_{\kappa}(\nabla(\ff)). $$
\end{lma}

\begin{proof}
Let $\mu$ be a homogeneous syzygy of degree $\kappa$. Then $\nabla(\ff)\circ \mu = 0$, so the map $\mu : \calo_{X}(-\kappa)\to \bigoplus_{i=1}^r\calo_{X}(D_i)$ factors through $\calt_\Sigma\langle \ff\rangle_0$ and hence provides an element of $H^0(\calt_\Sigma\langle \ff\rangle_0(\kappa))$. The converse follows the same argument.
\end{proof}

\begin{dfn}[Coefficient matrix] \label{coefficient matrix} Consider a vector of $p$ homogeneous syzygies 
$\nu=(\nu_1,\dots \nu_p)$ of $\nabla(\ff)$. Then the \textit{coefficient matrix} of $\nu$ is
$$M = M(\nu):=\left(\begin{array}{cccccc}
   \nu_{1,1}&\dots &\nu_{1,p}  \\
   \vdots  & \dots & \vdots     \\
   \nu_{r,1}&\dots  & \nu_{r,p} 
   \end{array}\right)
   $$
where  $\nu_{i,j}$ are the coefficents of the syzygy $\nu_j$, for $1 \le k \le p$ and $1 \le i \le r$.
\medskip

Fix $\phi_1,\dots, \phi_{\rho}$, a basis of $\Cl(X)^{\vee}$. With respect to this basis, we set, for $1 \le i \le \rho$:
$$\epsilon_i = \left(\begin{array}{c}
     \phi_1([D_1]) x_1   \\
     \vdots \\
    \phi_r([D_r]) x_r
\end{array}\right), \qquad \epsilon = (\epsilon_1,\ldots,\epsilon_\rho); $$  
these are the coefficients of Euler derivations associated with $\phi_1,\ldots,\phi_\rho$. We then consider the $r \times r$ matrix of elements of $S$ given by (recall that $n+\rho=r$):
$$ M(\nu|\epsilon):=\left(\begin{array}{cccccc}
   \nu_{11}&\dots &\nu_{1,n} & \phi_1([D_1])x_1&\cdots &\phi_{\rho}([D_1]) x_1\\
   \vdots  & \dots & \vdots & \vdots &\cdots&\vdots   \\
   \nu_{r,1}&\dots  & \nu_{r,n} & \phi_1([D_r])x_r &\cdots &\phi_{\rho}([D_r]) x_r
   \end{array}\right). $$
Note that $M(\nu|\epsilon)$ can be regarded as a morphism of sheaves
$$ \left( \bigoplus_{1 \le i \le n} \calo_X(-\kappa_i) \right) \oplus \calo_X^{\rho} \longrightarrow \bigoplus_{1 \le i \le r} \calo_X(D_i) . $$
\end{dfn}

We can finally state our general freeness criterion for divisors in toric varieties, rephrased in terms of the syzygy matrix $M(\nu|\epsilon)$.

\begin{thm} \label{thm:saito}
Let $X=X_\Sigma$ be a simplicial toric variety with no torus factors, and let $D =\V(f)\subset X$ be a reduced divisor of class $\beta \in \Cl(X)$.
\begin{enumerate}[label=(\roman*)]
\item \label{saito-i} The sheaf $\calt_\Sigma \langle D\rangle$ is free with exponents $(0^\rho,\kappa_1,\ldots,\kappa_{n})$ if and only if there are $n$-syzygies $\nu=(\nu_1,\dots \nu_{n})$ of $\nabla(f) $ such that $\det (M(\nu|\epsilon)) = cf$ with $c\in \C^*.$
\item \label{saito-ii} Given syzygies $\nu$ as above, we have $\det (M(\nu|\epsilon)) = cf$, with $c\in \C^*$, if,
\[\beta_0 + \sum_{1 \le j \le n} \kappa_j - \beta \quad \mbox{is not effective}
\]
where $\beta_0$ is the sum of classes of all toric divisors, i.e., the anticanonical class of $X$.
\item \label{saito-iii}
If, in addition, $H^0(\calo_{X}(-\kappa_i))=0$ for $1 \le i \le n$, then 
$$\calt_X \langle D\rangle \simeq \calo_X(-\kappa_1) \oplus \cdots \calo_X(-\kappa_n).$$
\end{enumerate}
\end{thm}

The proof relies on a technical result regarding reflexive sheaves which might be of independent interest. An easy example goes as follows.

\begin{ex}[Hypercube arrangement] \label{hypercube}
    Let $n \ge 1$ be an integer and let $X$ be the product of $n$ copies of $\PP^1$. Then the Cox ring $S$ is $\C[x_1,y_1,\ldots,x_n,y_n]$, where $x_i,y_i$ generate the vector space of global sections of the line bundle $\calo_X(F_i)$ obtained as pull-back of $\calo_{\PP^1}(1)$ via the projection $\pi_i: X \to \PP^1$ onto the $i$-th factor, for $1 \le i \le n$. Let $f_1,\ldots,f_n$ be homogeneous polynomials of degree $d_1,\ldots,d_n \ge 0$, with $f_i=f_i(x_i,y_i)$ for $1 \le i \le n$ and set $f=f_1\cdots f_n$, $d=d_1\cdots d_n$ and $D=\VV(f)$. Then we get
    \[
    \calt_\Sigma\langle D\rangle \simeq \bigoplus_{i=1}^n \calo_X(-d_i F_i) \oplus \calo_X^{\oplus n},
    \]
    and, if $d_i \ge 2$ for all $i \in \{1,\ldots,n\}$, then:
    \[
    \calt_X\langle D\rangle \simeq \bigoplus_{i=1}^n \calo_X(-d_i F_i) .
    \]
Indeed, we have syzygies $\nu_1,\ldots,\nu_n$ are expressed by column matrices having $n$ blocks of size $2$, the $i$-th block being (for $d_i \ge 1$) the transpose of 
\[
\left(\frac{\partial f_i}{\partial y_i}, \quad -\frac{\partial f_i}{\partial x_i}\right).
\]
Hence, considering the Euler syzygies $\epsilon_i=(x_i, y_i)$ for $1 \le i \le n$, up to rearranging the entries of $M(\nu| \epsilon)$ we get a block-diagonal matrix of $n$ blocks of size 2, the $i$-th block being
\[
\begin{pmatrix}
    \frac{\partial f_i}{\partial y_i}  & x_i\\ 
    -\frac{\partial f_i}{\partial x_i} & y_i
\end{pmatrix}
\]
The determinant of each such block is $d_i f_i$, hence $\det(M(\nu|\epsilon))=d f$ and the theorem applies.
\end{ex}

\subsection{A general lemma on reflexive sheaves}

Let us highlight a simple and general statement affording isomorphisms of reflexive sheaves that will be useful later on.

\begin{lma} \label{lem:iso1}
Let $W$ be a normal integral scheme, and let $\theta: \calf \hookrightarrow \calt$ be a monomorphism between reflexive sheaves on $W$, with $rk(\calf)=rk(\calt)$. If $\calodim(\V(\det(\theta)))\geq2$, or alternatively if $c_1(\calt)-c_1(\calf)$ is not effective, then $\theta$ is an isomorphism.
\end{lma}

\begin{proof}
Let $Y:=\V(\det(\theta))=\supp(\caloker(\theta))$. The assumption is that $Y$ has codimension at least two in $W$. Note that, since the divisorial part of $\caloker(\theta)$ has class $c_1(\calt)-c_1(\calf)$, this is guaranteed by the fact that such a divisor class is not effective. 

Then, for any open subset $U\subseteq W$, we have an exact sequence
$$ 0 \lra \calf(U) \stackrel{\theta(U)}{\lra} \calt(U) \lra \caloker(\theta(U)) $$
If $U\cap Y=\emptyset$, then $\theta(U):\calf(U) \to \calt(U)$ is an isomorphism. If $V:=U\cap Y\neq\emptyset$, then we use the fact that every reflexive sheaf on a normal integral scheme is also normal (see \cite[Proposition 1.6]{Hartshorne1980} and the definition right before it) to obtain the sequence of isomorphisms
$$ \calf(U) \simeq \calf(U\setminus V) \stackrel{\theta\big(U\setminus V\big)}{\lra} \calt(U\setminus V) \simeq \calt(U); $$
the morphism in the middle is an isomorphism because $\caloker\theta(U\setminus V)=0$; the leftmost and rightmost identifications come from the hypotheses that $\calf$ and $\calt$ are reflexive sheaves.

Thus $\theta(U)$ is an isomorphism for every open subset $U$ of $W$, so it must be an isomorphism.
\end{proof}

Note that the previous lemma does not require $W$ to be complete although we are working with complete toric varieties not necessarily projective.
For examples of non-projective complete toric varieties, see  \cite[Example 6.1.17]{CoxLittleSchenck}.

\subsection{Proof of the Saito criterion for divisors in toric varieties}

Assume first $\Cl(X) \ne 0$, so that there is $\phi$ such that Proposition \ref{T0} applies. Then  and we can work with $\calt_\Sigma \langle D\rangle_0$ and prove that it is free with exponents  $(0^{\rho-1},\kappa_1,\ldots,\kappa_{n})$. Recall the notation $\cali_{Z/X}(\beta)$ for the ideal of the Jacobian subscheme of $f$ in $X$ already used in the proof of Proposition \ref{T0}.
Set $\calf = \bigoplus_{1 \le i \le n} \calo_X(-\kappa_i)$.

Let $\nu=(\nu_1,\ldots,\nu_n)$ be Jacobian syzygies of degrees $(\kappa_1,\ldots,\kappa_n)$. Then we have a commutative diagram
\[
\xymatrix@-1.5ex{
& 0 \ar[d]& 0 \ar[d]& 0 \ar[d]\\
0 \ar[r] & \calf \oplus \calo_X^{\oplus\rho - 1} \ar^{\theta}[d] \ar[r] & \calf \oplus \calo_X^{\oplus\rho} \ar^M[d] \ar[r] & \calo_X \ar^f[d] \ar[r] & 0 \\
0 \ar[r] & \calt_\Sigma \langle D\rangle_0 \ar[r]\ar[d] & \bigoplus_{1 \le i \le r} \calo_X(D_i) \ar[d]\ar^-{\nabla(f)}[r] & \cali_{Z/X}(\beta) \ar[d]\ar[r] & 0 \\
0 \ar[r] & \calg \ar[r] \ar[d]& \caloker(M) \ar[d]\ar^{\vartheta}[r] & \cali_{Z/D}(\beta) \ar[d]\ar[r] & 0 \\
& 0 & 0 & 0
}
\] 

The monomorphism $\theta$ appearing in the leftmost column is induced by the diagram. We show that $\theta$ meets the requirements of Lemma \ref{lem:iso1} if the hypothesis of \ref{saito-i} or \ref{saito-ii} are fulfilled so that $\theta$ is an isomorphism, or equivalently $\calg=0$.

Note that the codimension-1 part $C$ of $\caloker(M)$ is a hypersurface of $X$ defined by the equation $C=\V(\det(M))$ and that $C$ is a divisor of class $\beta_0+ \sum_{1 \le j \le n} \kappa_j$ . Further:
\[
c_1(\calt_\Sigma \langle D\rangle_0)=-\beta + \beta_0, \qquad c_1(\calf \oplus \calo_X^{\oplus\rho-1})=-\sum_{1 \le j \le n} \kappa_j,
\]
so the assumption \ref{saito-ii} already implies that $\calt_\Sigma \langle D\rangle_0$ is free with the required exponents. 

On the other hand, if $\calt_\Sigma \langle D\rangle$ is free with the desired exponents, we get the syzygies $\nu$ directly and $\caloker(M) \simeq \cali_{Z/D}(\beta)$ so $C$ is the support of $\cali_{Z/D}(\beta)$, hence $C=D$, so that $\det(M)=c f$ which $c \in \C^*$. So \ref{saito-ii}  and the converse of \ref{saito-i} are proved.

To finish the proof of \ref{saito-i}, assume that $\det(M)=c f$ with $c \in \C^*$, so $C=D$. To apply Lemma \ref{lem:iso1} we need to see that the support $Y$ of $\calg$ has codimension at least $2$ in $X$. Note that $Y \subset C$ and that $C$ is reduced so we have to exclude that $Y$ contains an irreducible component of $D$. By contradiction, let $D_0$ be an irreducible component of $D$ contained in $Y$. Set $D'$ for the union of the irreducible components of $D$ distinct from $D_0$ and consider the Zariski-open subset $U$  of $X$ defined as the smooth locus of $X \setminus D'$.
Restricting the diagram to $U$, the reflexive sheaves under consideration become locally free, hence $\caloker(M|_U)$ has projective dimension $1$ and $\det(M|_U)=f|_U$ vanishes with multiplicity $1$ on $D_0 \cap U$, so $\caloker(M|_U)$ is torsion-free (actually locally Cohen--Macaulay) of rank $1$ on $D_0 \cap U$. Since the restriction of $\cali_{Z/D}(\beta)$ to $U$ is also torsion-free of rank $1$ on $D_0 \cap U$ and the map $\vartheta |_U$ induced by the diagram is surjective, its kernel, namely $\calg|_U$, is actually zero, so $Y$ does not meet $U$ and therefore does not contain the whole $D_0$, a contradiction.
Hence the support of $\calg$ has no divisorial components and the proof of \ref{saito-i} is finished.

If $\Cl(X)=0$, the argument works almost verbatim up to replacing $\calo_X^{\oplus\rho-1}$ with $\calo_X^{\oplus\rho}$ and $\calt_\Sigma \langle D\rangle_0$ with $\calt_\Sigma \langle D\rangle$. Indeed, in this case, we get a diagram as follows:
\[
\xymatrix@-1.5ex{
& 0 \ar[d]& 0 \ar[d]& \\
& \calf \oplus \calo_X^{\oplus\rho} \ar^{\theta}[d] \ar@{=}[r] & \calf \oplus \calo_X^{\oplus\rho} \ar^M[d]  \\
0 \ar[r] & \calt_\Sigma \langle D\rangle_0 \ar[r]\ar[d] & \bigoplus_{1 \le i \le r} \calo_X(D_i) \ar[d]\ar^-{\bar \nabla(f)}[r] & \cali_{Z/D}(\beta) \ar@{=}[d]\ar[r] & 0 \\
0 \ar[r] & \calg \ar[r] \ar[d]& \caloker(M) \ar[d]\ar^{\vartheta}[r] & \cali_{Z/D}(\beta) \ar[r] & 0 \\
& 0 & 0 & 
}
\]
Working with the same open subsets and using again that $\vartheta$ is surjective, we conclude once more that $\calg$ has no divisorial components and that Lemma \ref{lem:iso1} applies to show \ref{saito-i}.

To check \ref{saito-iii}, just note that if $\calt_\Sigma\langle D\rangle \simeq \calo_X^{\oplus\rho} \oplus \calf$ and $H^0(\calf)=0$, then by Proposition \ref{T vs TSigma} we must have $\calt_X\langle D\rangle \simeq \calf$.
This completes the proof. \hfill$\Box$

\subsection{Cones in weighted projective space} If $D\subset\PP^n$ is a free divisor, then it is easy to check that a cone $\hat D$ on $D$ in $\PP^{n+1}$ is also free. We will now check that a similar claim is true for weighted projective spaces.

To be precise, let $X=\PP[w_0,\dots,w_n]$ with Cox ring $S=\C[x_0,\dots,x_n]$; recall that $\rho(X)=1$. Take $f\in S$ and assume that $D=\V(f)$ is free with exponents $(\kappa_1,\dots,\kappa_n)$, i.e.
$$ \calt_{X}\langle D\rangle=\bigoplus_{i=1}^n \calo_{X}(\kappa_i) \quad {\rm and} \quad \calt\langle D\rangle = \calo_{X} \oplus \bigoplus_{i=1}^n \calo_{X}(\kappa_i). $$ 

Considering now $f$ as a polynomial on $\hat S=\CC[x_0,\dots,x_n,x_{n+1}]$ regarded as the Cox ring of $\hat X=\PP[w_0,\dots,w_n,w_{n+1}]$, set $\hat D=\V(f)$ as a divisor on $\hat X$. We will argue that $\hat D$ is free with exponents $(0,\kappa_1,\dots,\kappa_n)$.

Since $D\subset\PP[w_0,\dots,w_{n}]$ is free, Theorem \ref{thm:saito}(i) implies that we can find $n$ syzygies $\nu_1,\dots,\nu_{n}$ for $\nabla(f)$ such that $\det(M(\nu_1,\dots,\nu_{n}|\epsilon)) = cf$ for some $c\in\C^*$. When $f$ is regarded as a polynomial on $\hat S$, we have that $\partial f/\partial x_{n+1}=0$, thus $\nu_{n+1}=(0,\dots,0,1)$ is an additional syzygy for $\nabla(f)$, and the new coefficient matrix is given by
$$ \hat M(\nu_1,\dots, \nu_{n}|\epsilon)  = 
\left(\begin{array}{ccc}
& 0 & \phi([D_1])x_1\\
M(\nu_1,\dots,\nu_{n-1})& \vdots & \vdots \\
& 0 & \phi([D_{n-1}])x_{n-1} \\
0 \cdots 0 & 1 & \phi([D_n])x_n
\end{array}\right).$$
It is then easy to check that $\det(\hat M)=\det(M)=cf$, so Theorem \ref{thm:saito}(i) implies that 
$$ \calt_{\hat\Sigma}\langle D\rangle = \calo_{\hat X}^{\oplus 2} \oplus \bigoplus_{i=1}^n \calo_{\hat X}(\kappa_i)
\quad {\rm and} \quad
\calt_{\hat X}\langle D\rangle=\calo_{\hat X} \oplus\bigoplus_{i=1}^n \calo_{\hat X}(\kappa_i), $$ 
as desired.

\subsection{Toric braid arrangements}
The braid arrangement is a classic example of hyperplane arrangement on projective or affine spaces, see \cite[Example 1.9]{orlik1992arrangements} for more details. We will now introduce a divisorial arrangement on simplicial toric varieties that generalizes the braid arrangement.

Let $X=X_\Sigma$ be a simplicial toric variety with Cox ring $S=\C[x_1,x_2, \dots, x_r]$ and let $\mathcal{L} = \{ x_1, \ldots , x_r \}$ be the
set of generators of $S$. Note that $\mathcal{L}$ has a natural partition $\mathcal{L}=\mathcal{L}_1\bigcup \dots \bigcup \mathcal{L}_s$
where every element $l\in \mathcal{L}_i$, $1\leq i \leq s$ has the same degree. Let $r_i$ be the cardinality of $\mathcal{L}_i$ for $1\leq i \leq s$,  and, without loss of generality, let us assume that $\mathcal{L}_1=\{ x_1,\dots, x_{r_1}\}$, $\mathcal{L}_2=\{ x_{r_1+1},\dots, x_{r_1+r_2}\}$, $\dots$, $\mathcal{L}_s=\{ x_{r_{s-1}+1},\dots, x_{r_{s-1}+r_s}\}$.

With these definitions in mind, consider the following homogeneous polynomial in $S$:
$$ b_{\Sigma}:=\prod_{1\leq i<j\leq r_1} (x_i-x_j)\prod_{r_1+1\leq i<j\leq r_1+r_2} (x_i-x_j)\, \,\,\,\cdots \prod_{r_{s-1}+1\leq i<j\leq r_{s-1}+r_s}  (x_i-x_j), $$
or equivalently, 
$$ b_{\Sigma}:=\prod_{k=1}^s\prod_{{\underset{x_i,x_j\in \mathcal{L}_k}{i<j} }} (x_i-x_j). $$
The \textit{toric braid arrangement} is the divisor defined as $B_\Sigma:=\V\big(b_\Sigma\big)$. Let us consider some concrete examples:
\begin{enumerate}[label=\arabic*.]
\item When $X$ is the affine or the projective space, one recovers the usual braid arrangement.
\item When $X$ is the product of projective spaces of dimension $\PP^{n_1}\times \cdots \times \PP^{n_p}$, the Cox ring is 
$S=\C[x_{0,1},\ldots,x_{0,n_1},\ldots,x_{p,1}, \ldots,x_{p,n_p}]$. Its braid arrangement is given by the polynomial 
$$
b_{\Sigma}=\prod_{0 \le j_1 < i_1\le n_1} (x_{i_1,1}-x_{j_1,1})
\cdots \prod_{0 \le j_p < i_p\le n_p} (x_{i_p,p}-x_{j_p,p}).
$$
This arrangement is free with:
\[
\calt_\Sigma \langle D \rangle \simeq \bigoplus_{1 \le i \le p} \bigoplus_{0 \le j \le n_i} \calo_X((1-j) F_i)
,\]
where for $1 \le i \le p$, $F_i$ is the pull-back of the hyperplane divisor of $\PP^{n_i}$ under the projection of $X$ onto $\PP^{n_i}$.

\item When $X$ is the weighted projective space $\Pj(w_0,\ldots,w_n)$ with $\gcd(w_0,\ldots,w_n)=1$, the Cox ring is $\C[x_0,\ldots,x_n]$ with $\deg(x_i)=w_i$. Then its braid arrangement is just $b_{\Sigma}=\V(x_0 \cdots x_n)$ and we get:
\[
\calt_\Sigma \langle D \rangle \simeq \calo_X^{\oplus n+1}, \qquad \calt_X\langle D \rangle \simeq \calo_X^{\oplus n}.
\]
\end{enumerate}

To state the result, set $r_0=0$ and, for all $1 \le i \le s$ and any integer $j$ with $r_{i-1}+1 \le j \le r_i$, put $\kappa_i = \deg(\VV(x_j))$.
Note that this is consistent with our assumption on the degree of the variables $x_1,\ldots,x_r$. Then we have the following.

\begin{prop} \label{prop:braid}
Let $X_\Sigma$ be a simplicial toric variety with no torus factors.
Then  
\[
\calt_\Sigma\langle b_{\Sigma}\rangle \simeq \bigoplus_{1 \le i \le s} \bigoplus_{1 \le j \le r_i} \calo_X((1-j)\kappa_i).
\]
\end{prop}

\begin{proof} Without loss of generality one assumes that the cardinality of $\mathcal{L}_i$ is greater or equal to $\mathcal{L}_j$ if and only if $i\leq j $, i.e.,  $r_i\geq r_j$ for $i\leq j$. 
Now, for every $i\in \{1,\dots s\}$ and with $r_0=0$  one can consider
$$
M_{r_i}=
\left(\begin{array}{clll}
   1& \phi_{r_i}([D_1])x_{r_i+1} & \cdots & x_{r_i+1}^{r_i-1}\\
   1  & \phi_{r_i}([D_2])x_{r_i+2} & \cdots & x_{r_i+2}^{r_i-1} \\
   \vdots & \vdots  &\cdots & \vdots \\
   1&\phi_{r_i}([D_{r_i+r_{i+1}}]) x_{r_i+r_{i+1}} &\cdots & x_{r_i+r_{i+1}}^{r_i-1}
   \end{array}\right)_{r_i\times r_i}
$$
which it is a Vandermonde matrix up to the constant $\phi_{r_i}([D]_{r_i})=\dots =\phi_{r_i}([D]_{r_i+r_{i+1}})\in \Z $. This matrix has determinant  $\prod_{r_{k-1}+1\leq i<j\leq r_{k-1}+r_k}  (x_i-x_j)$. Hence considering $M$ as the block diagonal matrix made of the matrices $M_{r_i}$ plus, possibly, a lower triangular block matrix with zero block diagonal matrix, one can apply the toric Saito criterion of Theorem \ref{thm:saito} to obtain the required result.
\end{proof}

\subsection{Reduced invariant divisors} \label{sec:inv}
One can also use other coefficient matrices, with more than $n$ syzygies and less than $\rho$ Euler derivations. Let us consider a concrete example, in which only one Euler derivation is necessary.

Let $X=X_\Sigma$ a simplicial toric variety with Cox ring $S=\C[x_1,\dots, x_r]$, and recall that a divisor $D\subset X$ is invariant under the torus action if and only if (up to the order of the variables $x_1$) $D=\V(x_1\cdots x_s)$ where $s\leq r$.  We show that:

\begin{prop}\label{prop:inv-div}
Let $X=X_\Sigma$ be a simplicial toric variety with no torus factors.
The toric logarithmic tangent sheaf ${\calt}_\Sigma\langle D\rangle_0$ associated to the invariant divisor $D=\V(x_1\cdots x_s)$ is free:
$$ \calt_\Sigma\langle D\rangle_0 \simeq \calo_{X}^{\oplus s}\oplus \bigoplus_{i=s+1}^r \calo_{X}(D_i).$$
\end{prop}

\begin{proof}
We have $s-1$ syzygies of the form
$$ \nu_1=(x_1,-x_2,0,\dots,0) 
\quad\dots\quad \nu_{s-1}=(0,\dots,0,x_{s-1},-x_s,0,\dots, 0). $$ 
Ignoring the zeros in the last $s-1$ entries We then put together the $s\times(s-1)$ coefficient matrix $M_s:=M(\nu_1,\dots,\nu_{s-1})$, which can be regarded as a morphism of sheaves
$$ M_s: \calo_{X}^{\oplus s-1} \hookrightarrow \bigoplus_{i=1}^{s} \calo_{X}(D_i). $$
Adding a $(r-s)\times(r-s)$ identity matrix regarded as the morphism
$$ \mathbf{1}_{r-s}: \bigoplus_{i=s+1}^r \calo_{X}(D_i) \hookrightarrow \bigoplus_{i=s+1}^r \calo_{X}(D_i). $$
Finally, choose a nontrivial Euler derivation $\phi\in\Cl(X)$, and complete the $r\times(r-1)$ block diagonal matrix $M_s\oplus\mathbf{1}_{r-s}$ to a square matrix by adding the column $\epsilon=(\phi([D_1])x_1,\dots,\phi([D_{r}]) x_r)$.
We end up with the $r\times r$ coefficient matrix
$$ M(\nu|\mathbf{1}_{r-s}|\epsilon) : = \left(\begin{matrix}
M_s & 0 & \epsilon \\ 0 & \mathbf{1}_{r-s} & 
\end{matrix} \right) $$
whose first $r-1$ columns are syzygies for $\nabla f$. It defines a morphism of sheaves
$$ M(\nu|\mathbf{1}_{r-s}|\epsilon) : \calo_{X}^{\oplus s-1} \oplus \bigoplus_{i=s+1}^r \calo_{X}(D_i) \oplus \calo_{X} \longrightarrow \bigoplus_{i=1}^r \calo_{X}(D_i). $$

Next, note that
\begin{align*}
\det\big( M(\nu|\mathbf{1}_{r-s}|\epsilon)\big) = & \det
\begin{pmatrix}
-x_1 & 0 &\dots & 0& \phi([D_1])x_1  \\
x_2 & -x_2 & \dots &0& \phi([D_2])x_2  \\
0 & x_3 &\dots &0& \phi([D_3])x_3 \\
\vdots & \vdots & \ddots& \vdots &\vdots \\
0 & 0  & \dots  & -x_{s-1} & \phi([D_{s-1}])x_{s-1}\\
0 & 0  &\dots& x_s & \phi([D_{s}]) x_s
\end{pmatrix} \\
= & \det \begin{pmatrix}
-x_1 & 0 & \dots & 0& 0  \\
x_2 & -x_2 & \dots &0& 0  \\
0 & x_3  & \dots &0& 0 \\
\vdots & \vdots & \ddots& \vdots &\vdots \\
0 & 0 & \dots  & -x_{s-1} & 0\\
0 & 0 &\dots& x_s & c_{\phi}. x_s
\end{pmatrix}
= (-1)^{s-1}c_{\phi} x_1\cdots x_s ,
\end{align*}
where $c_{\phi}\in\C$ is a constant depending on the coefficients $\phi([D_i])$; the equality from the first to the second lines comes from inductively adding $\phi([D_1])$ times the first column with the last column, and so on always using the coefficient left in the $i^{\rm th}$-entry of the last column to multiply the $i^{\rm th}$-column and adding it back into the last column.

One can then use the same argument as in the proof of Theorem \ref{thm:saito} to show that the induced map
$$ \theta : \calo_{X}^{\oplus s-1} \oplus \bigoplus_{i=s+1}^r \calo_{X}(D_i) \longrightarrow T_\Sigma\langle D\rangle_0 $$
is an isomorphism.
\end{proof}

An immediate consequence of Proposition \ref{prop:inv-div} is that if a reduced divisor $D\subset X$ is an invariant divisor, then the toric logarithmic tangent sheaf $\calt_{\Sigma}\langle D\rangle_0$ is also invariant under the torus action. For comparison, Napame showed that a logarithmic tangent sheaf $\calt_{X}\langle D\rangle$ for a reduced divisor $D$ is torus invariant if and only if $D$ is an invariant divisor \cite[Proposition 3.3]{Napame}.

Napame has also studied in \cite{Napame} the slope-stability of logarithmic tangent sheaf $\calt_{X}\langle D\rangle$ for an invariant divisor in a projective toric variety $X$. Recall that a torsion-free sheaf $\mathcal{E}$ on a (possibly singular) complex projective variety $X$ is slope-(semi)stable with respect to a polarization $L$, if for any proper coherent subsheaf $\mathcal{F}$ of $\mathcal{E}$ with $0 < rk(\calf) < rk(\cale)$, one has $\mu_L (\mathcal{F}) < \mu_L(\mathcal{E} )$ (resp. $\mu_L(F) \leq  \mu_L(E )$)
where the slope $\mu_L$ of $\mathcal{E}$ with respect to $L$ is
$$ \mu_L(\cale) =\frac{ c_1(\cale) \cdot L^{\dim(X)-1}}{ \rk (\cale)}. $$

Napame proved the following claim in \cite[Section 4.1]{Napame} using pure toric geometry techniques for toric invariant reflexive sheaves. In the context of the present paper, it is an immediate consequence of Proposition \ref{prop:inv-div} and Propositions \ref{T vs TSigma} and \ref{T0}.

\begin{coro} \label{corol stable} Let $X=X_\Sigma$ be a simplicial projective toric variety with Picard rank 1. If $D\subset X$ is an invariant divisor, then the associated logarithmic sheaf $\calt_{X}\langle D\rangle$ is never slope-stable with respect to any polarization on $X$.
\end{coro}

\begin{ex} \label{example stable}
Let us revisit Example \ref{ex:34} by considering the case 
$X=\PP\big( \op1\oplus\op1(1)\big)\stackrel{\pi}{\to}$ (i.e. $a_1=0$ and $a_2=1$) and the divisor $D\subset X$ consisting of a single fiber of $\pi$ ($k=1$ in the example); to be precise, we take $f=y$ in the Cox ring $\C[x,y,u_1,u_2]$, so $D=\V(y)$ and $\ox(D)\simeq \ox(F)$ (in the notation of Example \ref{ex:34}). The exact sequence in display \eqref{eq:sqctxd} then becomes
$$ 0 \lra \ox \lra \ox(H)^{\oplus2}\oplus\ox(F) \lra \calt_X\langle D\rangle \lra 0. $$
Napame has shown in \cite[Theorem 5.5]{Napame} that $\calt_X\langle D\rangle$ is slope-stable with respect to a polarization of the form $L=\pi^*\op1(a)\otimes\ox(b)$ with $0<a<b$. In fact, $\calt_X\langle D\rangle$ satisfies the following short exact sequence
$$ 0 \lra \ox(F) \lra \calt_X\langle D\rangle \lra \ox(2H-F) \lra 0 , $$
with the rightmost sheaf being precisely $T_\pi X$.

In particular, $\calt_X\langle D\rangle$ does not split as a sum of line bundle, even though it is a \textit{nontrivial extension} of line bundles; therefore, the condition on the number of lines imposed in Example \ref{ex:34} ($k\ge2$ in the example at hand) is indeed necessary for the freeness of a union of fibers in $X=\PP\big(\op1(a_1)\oplus\op1(a_2)\big)$.
\end{ex}

\section{Logarithmic tangent sheaves for several polynomials and foliations} \label{sec:dist}

On a simplicial toric variety $X=X_\Sigma$ with Cox ring $S=\C[x_1,\dots, x_n]$, let $\ff=(f_1,\dots, f_k)$ be a sequence of elements of $S$. For each $i \in \{1,\ldots,k\}$ write $\deg(f_i)=\beta_i\in\Cl(X)$. Recall that a set of polynomials $\{f_1,\dots, f_m\}\subset\C[x_1,\dots, x_n]$ is called algebraically independent if there is no non-zero polynomial $F$ such that $F(f_1, \ldots, f_m) = 0$.

\subsection{Toric logarithmic tangent sheaves for several polynomials}
Given a sequence $\ff=(f_1,\dots, f_k)\subset S$, let $\nabla(f_i)$ be the Jacobian matrix, and consider their union
$$ \nabla(\ff):=\left(\begin{array}{c}
  \nabla(f_1)     \\
   \vdots   \\
\nabla(f_k)      
\end{array}
\right) \quad {\rm and} \quad 
\overline\nabla(\ff):=\left(\begin{array}{c}
  \overline\nabla(f_1)     \\
   \vdots   \\
\overline\nabla(f_k)      
\end{array}
\right) $$
as the Jacobian matrices, which can be regarded as a morphism between sheaves, namely 
$$ \nabla({\ff}):\bigoplus_{i =1}^r \calo_{X}(D_{i})\longrightarrow \bigoplus_{i=1}^k \calo_{X}(\beta_i) \quad {\rm and} \quad
\overline\nabla({\ff}):\bigoplus_{i =1}^r \calo_{X}(D_{i})\longrightarrow \bigoplus_{i=1}^k \calo_{D_i}(\beta_i), $$
where $D_i:=\V(f_i)$. If $\ff$ is algebraically independent, then $ \nabla({\ff})$ has maximal rank $k$. 

\begin{dfn} \label{defn:toric-several}
The \textit{extended toric logarithmic tangent sheaf} and the \textit{toric logarithmic tangent sheaf} associated with a sequence $\ff$ as above are respectively defined to be
$$ \calt_\Sigma\langle \ff \rangle := \ker\Big(\overline\nabla(\ff)\Big) \quad {\rm and} \quad \calt_\Sigma\langle \ff \rangle_0 := \ker\Big(\nabla(\ff)\Big). $$
\end{dfn}

Note that $\rk\big(\calt_\Sigma\langle \ff \rangle\big)=r$, while $\rk\big(\calt_\Sigma\langle \ff \rangle_0\big)=r-k$ when $\ff$ is algebraically independent.

\begin{rmk} In certain special cases, the Definition \ref{defn:toric-several} recovers certain objects that the first two named authors and Vallès have previously studied. 
\begin{enumerate}[label=\roman*)]
\item When $X$ is the projective space $\Pj^n$, $\calt_\Sigma\langle \ff \rangle_0$ coincides with the logarithmic tangent sheaf for the algebraically independent sequence $\ff$ introduced by the first two named authors and Vallès in \cite{FaenziJardimValle}; this was denoted by $\calt_\ff$ in that reference.
\item When $X$ is the affine space $\C^n$, the sheaves $\calt_\Sigma\langle \ff \rangle$ and $\calt_\Sigma\langle \ff \rangle_0$ are the sheafification of the logarithmic tangent modules ${\rm Der}(f_1,\dots, f_k)$ and ${\rm Der}_0(f_1,\dots, f_k)$ introduced by the first two authors and Vallès in \cite[Section 6]{FJV1}.
\item When $k=1$, we simply recover the toric logarithmic tangent sheaves $\calt_\Sigma \langle D\rangle$ and $\calt_\Sigma \langle D\rangle_0$ introduced in Definition \ref{defn:log-sheaves} above.
\end{enumerate}
\end{rmk}

We complete this section by mentioning the following version of \cite[Lemma 2.5]{FaenziJardimValle} adapted to the toric context.

\begin{lma} \label{lem:intersection}
We have 
$$ \calt_\Sigma\langle \ff \rangle=\bigcap_{i=1}^k \calt_\Sigma \langle D_i\rangle 
\quad {\rm and} \quad
\calt_\Sigma\langle \ff \rangle_0=\bigcap_{i=1}^k \calt_\Sigma \langle D_i\rangle_0 $$
with the intersections taken as subsheaves of $\bigoplus_{i =1}^r \calo_{X}(D_{i})$. In particular, if each $f_i$ is torus invariant, then $\calt_\Sigma\langle \ff \rangle$ and $\calt_\Sigma\langle \ff \rangle_0$ are also torus invariant.
\end{lma}
\begin{proof}
The first claim just follows from the fact that $\ker\Big(\nabla(\ff)\Big) = \bigcap_{i=1}^k \ker\Big(\nabla(f_i)\Big)$, and similarly for the morphisms $\overline\nabla(\ff)$ and $\overline\nabla(f_i)$. Moreover, if each $f_i$ is torus invariant, then, as observed in Section \ref{sec:inv}, each $\calt_\Sigma \langle D_i\rangle$ and $\calt_\Sigma \langle D_i\rangle_0$ are also torus invariant; the second claim follows immediately. 
\end{proof}

\begin{rmk} 
If the polynomials $f_1,\ldots,f_k$ form a regular sequence and therefore define a $k$-codimensional complete intersection subscheme $Y:=\V(f_1,\ldots,f_k)\subset X$, then the toric logarithmic tangent sheaf $\calt_\Sigma\langle \ff \rangle_0$ is a subsheaf of the usual logarithmic sheaf $\calt_X\langle Y\rangle$ of vector fields on $X$ tangent to $Y$, which controls the locally trivial deformations of $Y$ in $X$, see \cite{sernesi}.
\end{rmk}

\subsection{Foliations on toric varieties}

Recall that a distribution of codimension $c$ on a (possibly singular) complex variety $X$ of dimension $n$ is a short exact sequence 
$$ \mathcal{D} \quad : \quad 0 \lra \calt_{\mathcal{D}} \stackrel{\nu}{\lra} \calt_X \lra \mathcal{N}_{\mathcal{D}} \lra 0, $$
where $\mathcal{N}_{\mathcal{D}}$ is a torsion-free sheaf of rank $c$; $\mathcal{N}_{\mathcal{D}}$ is called the \textit{normal sheaf}, while $\calt_{\mathcal{D}}$ is called the \textit{tangent sheaf} of ${\mathcal{D}}$. When $X$ is normal, then the tangent sheaf $\calt_{\mathcal{D}}$ is automatically reflexive.
The \textit{singular scheme} of $\cald$ is defined as follows: the maximal exterior power of the dual morphism $\nu^\vee:\Omega^1_{X}\to\calt_\cald^\vee$ gives a morphism $\Omega^{n-c}_{X}\to\det(\calt_\cald)^\vee$; its image is a twisted ideal sheaf $\cali_{Z/X}\otimes\det(\calt_\cald)^\vee$ for a subscheme $Z\subset X$, and this is the singular scheme of $\cald$. When $X$ is non-singular, $Z$ coincides, as a closed subset, with the singular locus of the normal sheaf $\mathcal{N}_{\mathcal{D}}$, namely the variety
$$ \{x\in X ~|~ (\mathcal{N}_{\mathcal{D}})_x \textrm{ not free} \} = \bigcup_{p=1}^n \supp \Big( \inext^p(\mathcal{N}_{\mathcal{D}},\ox) \Big) . $$ 

A \textit{foliation} is an integrable
distribution; that is, in the notation of the previous paragraph, for each $x\in X\setminus Z$, there is an open neighborhood $x\in U\subset X$ (in the Euclidean topology) and an analytic submersion $\phi:U\to \C^{m}$ such that $\calt_{\mathcal{D}}|_U \simeq \ker(d\phi)$; the fibers of $\phi$ glue together and define immersed analytic subvarieties called the leaves of $\cald$. Equivalently by Frobenius' theorem, the tangent sheaf $\calt_{\mathcal{D}}$ is closed under the Lie bracket of vector fields.

We suggest \cite{zbMATH07323410} as a general reference for distributions; in addition, we also refer to two recent papers about distributions and foliations on toric varieties, namely \cite{pena2024codimension} and \cite{WANG202370}.
When $X=\PP^n$, Muniz showed in \cite[Appendix A]{FaenziJardimValle} that the logarithmic tangent sheaf $\calt\langle\ff\rangle_0$ associated with a sequence of homogeneous polynomials $\ff$ coincides, up to a twist by $\mathcal{O}_{\mathbb{P}^n}(-1)$, with the tangent sheaf of a \textit{rational} foliation on $\PP^n$.

The goal of the present section is to provide a partial generalization of \cite[Proposition A.2]{FaenziJardimValle} to the context of toric varieties. More precisely, given a $k$-tuple of homogeneous elements $\ff=(f_1,\ldots,f_k)$ in the Cox ring $S$ of a simplicial toric variety $X=X_\Sigma$, we define the following morphism of abelian groups:
\begin{equation} \label{eq:degmap}
\mathbf{deg}_\ff : \ZZ^{\oplus k} \to \Cl(X), \qquad (a_1,\ldots,a_k) \mapsto \sum_{1 \le i \le k} a_i\deg(f_i).  
\end{equation}

Define the \textit{rank degree} of $\ff$ as the rank of this morphism and set $\beta_i:=\deg(f_i)$. If $q:=\rk(\mathbf{deg}_\ff)$, then $\deg(\ff)=(\beta_1,\ldots,\beta_k)$ determines $k$ points $[\beta_1],\ldots,[\beta_k]$ with integral coefficients in the projective space $\p{q-1}$ over $\Q$. If $X$ is projective, then $\beta_i \ne 0$ for all $i \in \llbracket 1,k\rrbracket =\{1,\ldots,k\}$. By definition of $q$, these points are contained in no hyperplane. We say that $\deg(\ff)$ \textit{satisfies the Cayley--Bacharach condition} if there is no hyperplane of $\p{q-1}$ containing $k-1$ points among $[\beta_1],\ldots,[\beta_k]$.
Note that this condition is always satisfied when $q=1$; in particular, this happens if $\rho = 1$. When $q=2$, the Cayley--Bacharach property holds as soon as $\deg(\ff)$ consists of at least $3$ non-proportional vectors in $\Z^2$.

In addition, the degree morphism in display \eqref{eq:degmap} induces, by dualization and tensoring with $\calo_{X}$, the following morphism of sheaves
$$\Phi_\ff: {\rm Cl}(X)^{\vee}\otimes_{\Z} \calo_{X}\lra \calo_{X}^{\oplus k}, \qquad \Phi_\ff(\phi\otimes g) := (\phi(\beta_1)\cdot g,\dots, \phi(\beta_k)\cdot g). $$
Recall that that $\Cl(X)$ is a finitely generated group of rank $\rho$  (possibly with torsion), so $\Cl(X)^{\vee}$ is free and 
$\Cl(X)^{\vee}\otimes_{\Z} \calo_{X}\simeq \calo_{X}^{\oplus \rho}$. Since $\im(\Phi_\ff)$ is a subsheaf of $\calo_{X}^{\oplus k}$ and a quotient of $\calo_{X}^{\oplus\rho}$, we get that $\im(\Phi_\ff)=\calo_{X}^{\oplus q}$. Furthermore, it also follows that $\ker(\Phi_\ff)\simeq\ox^{\oplus\rho-q}$.

\begin{thm}\label{dist}
Let $X=X_\Sigma$ be a projective simplicial toric variety, and let $\ff$ be a sequence of $k<n$ pairwise coprime algebraically independent polynomials of degree rank $q$ with  $k-n < q < k$. If $\deg(\ff)$ satisfies the Cayley--Bacharach condition, then $\ff$ induces a foliation $\cald_{\ff}$ of codimension $k-q$ on $X$. If $X$ is non-singular, the singular scheme of $\cald_{\ff}$ contains $\V(\ff)$.
\end{thm}

We will break down the proof into two lemmas. We will first construct the distribution underlying the induced foliation $\cald_{\ff}$. 

\begin{lma}
Under the hypothesis of Theorem \ref{dist}, $\ff$ induces a distribution $\cald_{\ff}$ of codimension $k-q$ on $X$. If $X$ is non-singular, the singular scheme of this distribution contains $\V(\ff)$.
\end{lma}

\begin{proof}
The Euler relation for $\phi\in 
\Hom(\Cl(X),\Z)$ gives
$$\sum_{j=1}^r\phi([D_{j}])x_{j}\frac{\partial f_i}{\partial x_{j}}= \phi(\beta_i)f_i, \qquad \text{for all $1\leq i \leq k$}.$$ 
Using such relation and the generalized Euler sequence in toric varieties, we construct the next commutative exact diagram, where 
$\epsilon$ is defined by $\phi \mapsto (\phi([D_1])x_1,\dots \phi([D_r])x_r)$ and $\diag(\ff)$ is the $k\times k$ diagonal matrix with entries $(f_1,\dots,f_k)$. We then have the following commutative diagram:
$$ \xymatrix@-1.5ex{  & 0\ar[d] & 0\ar[d]& 0 \ar[d] \\
0\ar[r] & \ker(\Phi_\ff)\ar[r] \ar[d] & {\rm Cl}(X)^{\vee}\otimes_{\Z} \calo_{X}\ar[d]_-{\epsilon}\ar[r]^-{\Phi_\ff} &  \calo_{X}^{\oplus k}\ar[d]^-{\diag(\ff)}\\
0\ar[r]&\calt_\Sigma\langle \ff \rangle_0\ar[r] &\bigoplus_{i=1}^r \calo_{X}(D_{i}) \ar[r]^-{\nabla(\ff)}& \bigoplus_{i=1}^k\calo_{X}(\beta_i) . } $$
Set $\calt_{\ff}:=\calt_\Sigma\langle \ff \rangle_0/\ker(\Phi_\ff)$ and note that this is a subsheaf of ${\rm coker}(\epsilon)=\calt_X$. Moreover, since $\diag(\ff)$ maps $\im (\diag(\ff))$ into $\im (\nabla(\ff))$, we set $\mathcal{N}_{\ff}:=\im (\nabla(\ff))/\im (\diag(\ff))$. We end up with the following commutative diagram:
\begin{equation} \label{eq:diag1}
\begin{split}
\xymatrix@-1ex{  & 0\ar[d] & 0\ar[d]& 0 \ar[d] &\\
0\ar[r] & \ker(\Phi_\ff) \ar[r] \ar[d] & {\rm Cl}(\Sigma)^{\vee}\otimes_{\Z} \calo_{X}\ar[d]_{\eta}\ar[r] &  \im (\Phi_\ff) \ar[d]\ar[r]&0\\
0\ar[r]&\calt_\Sigma\langle \ff \rangle_0\ar[r]\ar[d] &\bigoplus_{i=1}^r \calo_{X}(D_{i})\ar[d] \ar[r]^{\nabla(\ff)}& \im(\nabla(\ff)) \ar[d]\ar[r]&0 \\
0\ar[r]& \calt_{\ff}\ar[d]\ar[r] & \calt_{X}\ar[r]\ar[d] & \mathcal{N}_{\ff}\ar[d]\ar[r]&0\\
{}&{0}& 0& 0 }    
\end{split}
\end{equation}
Our goal is to show that the bottom line in the previous diagram defines a distribution of codimension $k-q$ on $X$.

Since $\ff$ is algebraically independent, the sheaf $\im (\nabla(\ff))$ has maximal rank $k$ by \cite[Theorem 2.3]{EhrenborgRota1993}; it follows that $\rk(\mathcal{N}_{\ff})=k-q$. In addition, $\calt_{X}$ is reflexive (because $X$ is normal), therefore, it is enough to prove that $\mathcal{N}_{\ff}$ is torsion-free. 

By hypothesis, $\ff$ has degree rank $q$, so there are $(\gamma_1,\ldots,\gamma_q) \in \Cl(X)^{\oplus q}$ and a rank-$q$ matrix $A=(a_{i,j})$, with $1 \le i \le k$ and $1 \le j \le q$ such that, for all $\phi \in \Cl(X)^\vee$, one has 
 \[\Phi_\ff(\phi)=A \begin{pmatrix}
     \phi(\gamma_1) \\
     \vdots\\
     \phi(\gamma_q)
\end{pmatrix}. \]
We have $q<k$ by assumption and recall that  $\im(\Phi_\ff)\simeq\calo_{X}^{\oplus q}$ and $\ker(\Phi_\ff) \simeq \calo_{X}^{\rho-q}$. Let $\tilde{\ff}$ be the restriction of $\diag(\ff)$ to $\im(\Phi_\ff)$; its matrix expression is then given by
\[ \diag(\ff) A = \begin{pmatrix}
a_{1,1}f_1 & \cdots & a_{1,q} f_1 \\
\vdots & \ddots & \vdots \\
a_{k,1}f_1 & \cdots & a_{k,q} f_k 
\end{pmatrix}. \]
We then obtain the following commutative diagram

\begin{equation} \label{eq:diag2}
\begin{split}
\xymatrix@-1ex{  & 0\ar[d] & 0\ar[d]&   &\\
& \calo_{X}^{\oplus q}\ar@{=}[r]  \ar[d] & \calo_{X}^{\oplus q}\ar[d]^-{\tilde\ff} & &\\
0\ar[r]&\im(\nabla(\ff))\ar[r]\ar[d] &\bigoplus_{j=1}^k \calo_{X}(\beta_j)\ar[d] \ar[r]& \caloker(\nabla(\ff)) \ar@{=}[d] \ar[r] &0 \\
0\ar[r]& \mathcal{N}_{\ff}\ar[d]\ar[r] & \caloker(\tilde{\ff})\ar[r]\ar[d] & \caloker(\nabla(\ff))\ar[r]&0\\
{}&{0}& 0&  }
\end{split}
\end{equation}

Now, we claim that $\caloker(\tilde \ff)$ is torsion-free, implying that $\mathcal{N}_{\ff}$ is also torsion-free. 
To check this, it suffices to show that the maximal minors of the matrix in the previous display have no common factor. Any maximal minor of this matrix is obtained by choosing $1 \le i_1 < \cdots < i_q \le k$ and considering the matrix $A_{i_1,\ldots,i_q}$ obtained by selecting the columns $(i_1,\ldots,i_q)$ of $A$: the minor will be 
 $$\det((\diag(\ff) A)_{i_1,\ldots,i_q}) = \det(A_{i_1,\ldots,i_q})f_{i_1}\cdots f_{i_q}$$.

Since $f_1,\ldots,f_k$ are pairwise coprime, to check that these minors have no common factor it suffices to show that, for any $j \in \{1,\ldots,k\}$, there is a non-vanishing minor $A_{i_1,\ldots,i_q}$, where the indices $1\le i_1,\ldots,i_q \le k$ lie in $\llbracket 1,k\rrbracket \setminus \{j\}$, as then $\det(A_{i_1,\ldots,i_q})f_{i_1}\cdots f_{i_q}$ is coprime to $f_j$. Since $\deg(\ff)$ satisfies the Cayley--Bacharach property of rank-degree $q$, for any choice of $j \in \llbracket 1,k \rrbracket$, the points of $\p {q-1}$ determined by $\{[\beta_i] \mid i \in \llbracket 1,k\rrbracket \setminus \{j\}\}$ span $\p {q-1}$, so there is a minor $\det(A_{i_1,\ldots,i_q})\ne 0$, where $i_1,\ldots,i_q \in \llbracket 1,k\rrbracket $ are distinct from $j$. This shows that the maximal minors of $\diag(\ff) A$ have no common factor.

In conclusion, $\mathcal{N}_{\ff}$ is torsion-free, so the bottom sequence of the diagram in display \eqref{eq:diag2} defines a distribution on $X$, namely:
$$ \mathcal{D}_{\ff}  : \quad 0\lra \calt_{\ff}\lra \calt_X \lra \mathcal{N}_{\ff} \lra 0. $$

For the second claim, assume now that $X$ is non-singular. The bottom sequence in diagram \eqref{eq:diag2} yields the short exact sequence for each $p=1,\dots,n$
$$ \inext^p(\caloker\tilde{\ff},\ox) \lra \inext^p(\mathcal{N}_{\ff},\ox) \lra \inext^{p+1}(\caloker(\nabla(\ff)),\ox). $$

Dualizing the middle column we notice:
\begin{align*}
  &\supp\big(\inext^1(\caloker\tilde{\ff},\ox)\big)=\V(\ff), \\
  &\inext^p(\caloker\tilde{\ff},\ox)=0, && \mbox{for $p>1$},\\
  &\calodim\big(\inext^{p+1}(\caloker(\nabla(\ff)),\ox)\big)\ge p+1.
\end{align*}
It follows that $\V(\ff)\subset\supp\big(\inext^1(\caln_{\ff},\ox)\big)$, as desired.
\end{proof}

Next, we complete the proof of Theorem \ref{dist} by showing that $\cald_{\ff}$ is integrable.

\begin{lma}
Under the hypothesis of Theorem \ref{dist}, there exists a toric
variety $Y$ and a rational map $\phi : X\dashrightarrow Y$ such that $\calt_\ff$ is tangent to its fibers.
\end{lma}

\begin{proof}
The matrix $A$ appearing in the proof of Theorem \ref{dist} defines an action of $T_q = (\CC^*)^q$ on $\CC^k$ by the weights given by coefficients of $A$. In other word,s we consider the torus $T_k = (\CC^*)^k \subset \CC^k$ and the morphism $T_q \to T_k$ defined by the matrix $(a_{i,j})$ as:
\[
(\lambda_1,\ldots,\lambda_q) \mapsto (\lambda_1^{a_{1,1}}\cdots \lambda_q^{a_{1,q}},\ldots,
\lambda_1^{a_{k,1}}\cdots \lambda_q^{a_{k,q}}).
\]

Then this morphism has a finite kernel, the associated map of Lie algebras being given by matrix $(a_{i,j})$, which has rank $q$ by definition. Then $T_q$ operates on $T_k$ and we get the expected action on $\CC^k$. This way, $\ff=(f_1,\ldots,f_k)$ defines a rational map:
\[
\phi : X \dashrightarrow T_k / T_q, \qquad x \mapsto (f_1(x),\ldots,f_k(x)).
\]

Indeed, let $U$ be the intersection of the smooth locus of $X$ with the complement of $\VV(f_1) \cup \cdots \cup \VV(f_k)$. Then for any $x \in U$ and any representative $x'$ of $x$, there is $\lambda = (\lambda_1,\ldots,\lambda_q) \in T_q$ such that 
\[
f_i(x')=\lambda_1^{a_{i,1}} \cdots \lambda_q^{a_{i,q}}f_i(x), \qquad \mbox{for all $i \in \llbracket 1,k\rrbracket.$}
\]

We argue that the fibers of $\phi$ are the leaves of a foliation whose tangent sheaf is $\calt_\ff$. Indeed, let $Y = T_k / T_q$. Then the differential of $\phi$ along $U$ gives a map
\[
\mathrm{d} \phi : \calt_U \to \phi^*(\calt_Y),
\]
and the kernel of $\mathrm{d}  \phi$ at a point $x \in U$ is the tangent space to the fiber of $\phi$ passing through $x$.

On the other hand, the tangent sheaf of $Y$ is the cokernel of the map $g:\calo_Y^{\oplus q} \to \calv$ where $\calv$ is the tautological bundle of rank $k$ on $Y$ given by the weights of $T_q$ acting on $T_k$.
By definition of $\phi$, the pull-back $\phi^*(\calv)$ is just $\bigoplus_{1 \le j \le k} \calo_X(\beta_i)$ and the pull-back of $g$ is just $\tilde \ff$.
Moreover, lifting $\calt_U$ to $\bigoplus_{1 \le i \le r} \calo_U(D_i)$, the map $\mathrm{d} \phi$ is expressed by $\nabla(\ff)$.
So $\mathrm{d} \phi$ fits into:
\[
\xymatrix@-1ex{
& \calo_U^{\oplus q} \ar[d] \\
\bigoplus_{1 \le i \le r}\calo_U(D_i) \ar^-{\tilde \ff }[d] \ar^-{\nabla \ff}[r] & \bigoplus_{1 \le j \le k} \calo_U(\beta_j) \ar[d] \\
\calt_U \ar^-{\mathrm{d} \phi}[r]& \phi^*(\calt_Y)
}
\]
Therefore, comparing with the diagrams in displays \eqref{eq:diag2} and \eqref{eq:diag1}, we conclude that $\ker(\mathrm{d} \phi)$ is isomorphic to $\calt_\ff$. It follows that the distribution $\mathcal{D}_{\ff}$ is indeed a foliation, with leaves given by the fibers of the rational map $\phi$.
\end{proof}

The following claim is an immediate consequence of Theorem \ref{dist}; it is a version of Proposition \ref{T vs TSigma} for sequences of homogeneous polynomials.

\begin{coro}\label{cor:sqc}
Assume that $X$ and $\ff$ satisfy the conditions of Theorem \ref{dist}. The following short exact sequence relates the logarithmic tangent sheaf $\calt_\Sigma\langle \ff \rangle_0$ and the tangent sheaf ${\calt}_{\ff}$ of the induced foliation $\mathcal{D}_{\ff}$:
$$0 \lra \calo^{\oplus \rho-q}_{X} \lra  \calt_\Sigma\langle \ff \rangle_0 \lra {\calt}_{\ff} \lra 0, $$
where $q$ is the rank degree of $\ff$. In particular, if $q=\rho(X)$ then $\calt_\Sigma\langle \ff \rangle_0={\calt}_{\ff}$.
\end{coro}

Note that the equality  $\calt_\Sigma\langle \ff \rangle_0={\calt}_{\ff}$ happens $\rho(X)=1$. 
Here is one example with a higher Picard rank.

\begin{ex}
Choose positive integers $n,m$ with $n+m \ge 3$. Set $X=\Pj^n\times \Pj^m$ with coordinates $x=(x_0,\ldots,x_n)$ and $y=(y_0,\ldots,y_m)$. Consider the sequence $\ff=(f_1,\ldots,f_k)$ in the following cases.
\begin{enumerate}[label=\roman*)]
\item If $k=2$ and $f_1(x,y)$ and $f_2(x,y)$ are non-constant bihomogeneous of the bidegree $(d_1,e_1)$ and $(d_2,e_2)$ with $d_1e_2\ne e_1d_2$, for instance $m \ge 2$, $f_1(x,y)=x_0y_0$ and $f_2(x,y)=x_1y_1y_2$, then $q=k=2$ and Theorem \ref{dist} does not apply.
\item In the same situation as before, if $d_1e_2=e_1d_2$, then $1=q<k=2$ and $\ff$ induces a foliation of codimension $k-q=1$.
For instance for $f_1(x,y)=x_0y_0$ and $f_2(x,y)=x_1y_1$, 
we get $\calt_\Sigma\langle \ff \rangle_0=\ox^{\oplus2}$.
Corollary \ref{cor:sqc} implies that ${\calt}_{\ff}=\ox$, so the foliation induced by the second sequence is given by the short exact sequence
$$ 0 \lra \ox \lra \calt_X \lra \cali_Z(2,2) \lra 0, $$
where $Z$ consists of 8 points, namely $\V(\ff)\cup\V\big(\bigwedge^2 \nabla(\ff)\big)$.
\item If $k=3$ and $f_1,f_2,f_3$ are non-constant bihomogeneous of bidegree $(d_1,e_1)$, $(d_2,e_2)$ and $(d_3,e_3)$ with $d_1e_2 \ne e_1d_2$, $d_1e_3 \ne e_1d_3$, and $d_2e_3\ne e_2d_3$, then we have $2=q<3=k$ and the three points $(d_1:e_1)$, $(d_2:e_2)$, $(d_3:e_3)$ of $\p 1$ satisfy the Cayley--Bacharach property, so Theorem \ref{dist} applies to give a foliation of codimension 1.
\item Take $n=3$, $m=4$, $k=3$ and $f_1=x_0y_0$, $f_2=x_1x_2y_1y_2$ and $f_3=x_3y_3y_4$ so the degrees are $(1,1)$, $(2,2)$ and $(1,2)$, hence $q=2<3=k$ but these degrees do not satisfy the Cayley--Bacharach property hence Theorem \ref{dist} does not apply. Note that the map $\tilde \ff$ can be expressed as:
\[\begin{pmatrix}
  x_0y_0  &  x_0y_0  \\
 2x_1x_2y_1y_2 & 2x_1x_2y_1y_2 \\
 x_3y_3y_4   &  2x_3y_3y_4
\end{pmatrix}
\]
The cokernel of this map has torsion along the divisor $\VV(x_3y_3y_4)$. In this case, the subsheaf $\calt_\ff$ of $\calt_X$ has determinant $\calo_X(-1,-1)$ and is not saturated, its saturation has determinant $\calo_X(0,1)$, the difference being given by the bidegree of $\VV(x_3y_3y_4)$.
\end{enumerate}
\end{ex}

\bibliographystyle{amsalpha-my}
\bibliography{references}
\end{document}